\documentclass[11pt]{article}
\usepackage[T1]{fontenc}
\usepackage{enumerate}
\usepackage{amsmath}
\usepackage{amsthm}
\usepackage{amsfonts}
\usepackage{amssymb}
\usepackage{algorithm}
\usepackage{algpseudocode}
\usepackage[numbers]{natbib}
\usepackage{xcolor}
\usepackage{graphicx}
\usepackage{tikz}
\usepackage{fullpage}
\usepackage{xcolor}
\usepackage{hyperref}
\usepackage{bbm}

\newtheorem{question}{Question}
\numberwithin{question}{section}

\newtheorem{theorem}[question]{Theorem}
\newtheorem{proposition}[question]{Proposition}
\newtheorem{lemma}[question]{Lemma}
\newtheorem{claim}[question]{Claim}
\newtheorem{fact}[question]{Fact}
\newtheorem{observation}[question]{Observation}
\newcommand*\samethanks[1][\value{footnote}]{\footnotemark[#1]}

\newtheorem{definition}[question]{Definition}

\newtheorem{remark}[question]{Remark}

\title{The threshold for the asymmetric vertex-Ramsey property in randomly perturbed graphs}
\author{Asier Calbet\thanks{Ume{\aa} Universitet, Sweden. Emails: \texttt{asier.calbet@umu.se}, \texttt{victor.falgas-ravry@umu.se}. Research supported by Kempe Stiftelse grant JCSMK24-556 and Swedish Research Council grant VR 2021-03687.} \and Victor Falgas-Ravry\samethanks \and Joseph Hyde\thanks{University of Birmingham, UK.  Email: \texttt{j.f.hyde@bham.ac.uk}. Research supported by a Leverhulme Early Career Fellowship.}}

\begin{document}

\maketitle
\begin{abstract}
For $r \geq 2$ and graphs $H_1, \ldots, H_r, G$, we say that $G$ is $(H_1, \ldots, H_r)$ vertex-Ramsey, or $(H_1, \ldots, H_r)_v$-Ramsey, if whenever we colour the vertices of $G$ with colours from the set $[r]=\{1,2, \ldots, r\}$ 
there exists $j \in [r]$ such that some copy of $H_j$ in $G$ is monochromatic in colour $j$.
Given any fixed collection of graphs $H_1, \ldots, H_r$, {\L}uczak, Ruci\'{n}ski and Voigt~\cite{lrv} and Kreuter~\cite{kreuter96} determined in the 1990s the threshold edge probability $p$ at which the binomial random graph $G(n,p)$ becomes $(H_1, \ldots, H_r)_v$-Ramsey. More recently, Das, Morris and Treglown~\cite{dmt} investigated the vertex-Ramsey property in the randomly perturbed setting. When $r=2$ they determined the number of random edges one must add to a dense graph to ensure that with probability $1-o(1)$ the resulting graph is $(H_1, H_2)_v$-Ramsey whenever one of $H_1$ or $H_2$ is a clique. They posed the problem of extending their results to all pairs of graphs $(H_1, H_2)$.

In this paper we resolve a more general form of their problem and determine for any $r\geq 2$ and $r$-tuple of graphs $(H_1, \ldots, H_r)$ the number of random edges one must add to a dense graph to ensure that with probability $1-o(1)$ the resulting graph is $(H_1, \ldots, H_r)_v$-Ramsey.
\end{abstract}
\section{Introduction}
Given $r \geq 2$ and graphs $H_1, \ldots, H_r, G$, we say that \emph{$G$ is $(H_1, \ldots, H_r)$ vertex-Ramsey}, or that \emph{$G$ is $(H_1, \ldots, H_r)_v$-Ramsey},
if whenever we colour the vertices of $G$ with colours from the set $[r]:=\{1,2, \ldots, r\}$ there exists some $j \in [r]$ such that some copy of $H_j$ in $G$ is monochromatic in colour $j$. We denote this property by $G \rightarrow (H_1, \ldots, H_r)_v$. In the special case when $H_1 = \cdots = H_r = H$ we write this more compactly as $G \rightarrow (H, r)_v$ and say `$G$ is $(H, r)_v$-Ramsey'.

In the classical setting, the vertex-Ramsey problem is fairly uninteresting compared with its edge counterpart.
Indeed, denoting by $v(G)$ the order of a graph $G$ and by $K_n$ the complete graph of order $n$, it is easy to see for any $r \geq 2$ and graphs $H_1, \ldots, H_r$ that $K_n \rightarrow (H_1, \ldots, H_r)_v$ if and only if $n\geq v(H_1) + \cdots + v(H_r) - (r - 1)$. More generally for dense but non-complete graphs, writing $e(G)$ for the number of edges in a graph $G$, we have the following. We say a graph is non-empty is it contains at least one edge. \begin{observation}\label{obs:dense}
    Let $r\geq 2$ and $H_1, \ldots, H_r$ be non-empty graphs.
    There exists a graph $G^{\star}$ with $$e(G^{\star}) \leq \left(1 - \frac{1}{\sum_{j=1}^{r}(\chi(H_j) - 1)}\right)\left(\binom{v(G^{\star})}{2} + \frac{v(G^{\star})}{2}\right)$$
    such that $G^{\star} \not\rightarrow (H_1, \ldots, H_r)_v$.
    On the other hand, for any fixed $\varepsilon > 0$ there exists $n_0 \in \mathbb{N}$ such that any graph $G$ with $v(G)\geq n_0$ and $e(G) \geq \left(1 - \frac{1}{\sum_{j=1}^r(\chi(H_j) - 1)} + \varepsilon\right)\binom{v(G)}{2}$ satisfies $G \rightarrow (H_1, \ldots, H_r)_v$.
\end{observation}
\begin{proof}
For the first part, consider a complete balanced $\sum_{j=1}^{r}(\chi(H_j) - 1)$-partite graph $G^{\star}$. This graph has
$e(G^{\star}) \leq \left(1 - \frac{1}{\sum_{j=1}^{r}(\chi(H_j) - 1)}\right)\frac{v(G^{\star})^2}{2}$ 
edges. For each colour $j\in [r]$, select a set of $\chi(H_j)-1$ parts of $G$ and colour every vertex within them in colour $j$. This clearly yields an $r$-colouring of the vertices of $G^{\star}$ containing no monochromatic copy of $H_j$ in colour $j$ for any $j\in [r]$.

The second part of the observation follows from the famous Erd\H{o}s--Stone--Simonovits~\cite{esimonovits, estone} theorem. Indeed, let $H^*$ be the complete $\left(1+\sum_{j=1}^{r}(\chi(H_j) - 1)\right)$-partite graph with all part sizes equal to $\left(1+\sum_{j=1}^r (v(H_j)-1)\right)$. Observe that $H^* \to (H_1, \ldots, H_r)_v$: by the pigeon-hole principle, for every part in $H^{\star}$, there is at least one colour $i$ such that at least $v(H_i)$ vertices in that part are in colour $i$. By the pigeonhole principle again, there is some colour $j$ for which $i=j$ for at least $(\chi(H_i) =) \ \chi(H_j)$ different parts, yielding a monochromatic copy of $H_j$ in colour $j$.

 Now for any $\varepsilon > 0$, the Erd\H{o}s--Stone--Simonovits theorem~\cite{esimonovits, estone} implies that there exists $n_0 \in \mathbb{N}$ such that for any graph $G$ with $v(G) \geq n_0$ with $e(G) \geq \left(1 - \frac{1}{\sum_{j=1}^r(\chi(H_j) - 1)} + \varepsilon\right)\binom{v(G)}{2}$, we have $H^* \subseteq G$. Any such graph $G$ is thus $(H_1, \ldots, H_r)_v$-Ramsey.
\end{proof}
The graph $G^{\star}$ constructed in Observation~\ref{obs:dense} is both dense and highly structured. It is thus natural to ask about the vertex-Ramsey properties of host graphs $G$ with more uniform edge distribution. 
In this vein, \L uczak, Ruci\'{n}ski and Voigt~\cite{lrv} initiated the study of vertex-Ramsey properties of random graphs. 
The model of random graphs considered here is the Erd\H{o}s-R\'{e}nyi binomial random graph model $G(n,p)$ with $n$ vertices and each edge included independently with probability $p = p(n)$. To state {\L}uczak, Ruci\'{n}ski and Voigt's main result, we must first introduce a key graph parameter.
\begin{definition}
    Given a graph $H$, we define the \emph{1-density of $H$} to be $$m_1(H) := \max\left\{\frac{e(J)}{v(J) - 1} : J \subseteq H, v(J) \geq 2\right\},$$setting $m_1(H) = 0$ when $v(H) \leq 1$.
\end{definition}
\noindent Note that throughout this paper, we follow the convention of setting $n^{-1/0} := 0$.
\begin{theorem}[\L uczak, Ruci\'{n}ski and Voigt~\cite{lrv}]\label{thm:lrv} For $r \geq 2$ and any graph $H$ with $e(H)\geq 1$, 
excluding the case when $r = 2$ and $H$ is a matching, the following holds: 
there exist constants $c, C>0$ such that $$\lim_{n\to \infty} \mathbb{P}[G(n,p) \to (H, r)_v] = \begin{cases}
    1 & \mbox{if} \ p > Cn^{-\frac{1}{m_1(H)}}    \\
    0 & \mbox{if} \ p < cn^{-\frac{1}{m_1(H)}}. 
    \end{cases}$$
\end{theorem}
\noindent We refer to the statements in the cases $\lim_{n\to \infty} \mathbb{P}[G(n,p) \to (H, r)_v] = 1$ 
and $\lim_{n\to \infty} \mathbb{P}[G(n,p) \to (H, r)_v] = 0$ in such results as the {\it 1-statement} and the {\it 0-statement}, respectively. 
When $H$ is a matching, for any constant $c > 0$ we have that many vertex disjoint copies of $K_3$
appear in $G(n,p)$ with probability at least some constant $c' = c'(c) > 0$ whenever $p \geq c n^{-1} = c n^{-1/m_1(H)}$. 
Observe that $K_3 \rightarrow (K_2, K_2)_v$.
Hence, we require $p = o(n^{-1/m_1(H)})$ rather than $p \leq cn^{-1/m_1(H)}$ for the $0$-statement when $H$ is a matching.

Kreuter~\cite[Theorem 2]{kreuter96} proved an asymmetric generalization of Theorem~\ref{thm:lrv}. 
To state Kreuter's result, we must introduce another graph parameter, refining the notion of 1-density.
\begin{definition}[Kreuter threshold]
    Given any two graphs $H_1$ and $H_2$ with $m_1(H_1) \geq m_1(H_2)$, 
    we define the \emph{mixed 1-density of $H_2$ and $H_1$} (or \emph{Kreuter threshold of $H_2$ and $H_1$}) to be $$m_1(H_2, H_1) := \max\left\{\frac{m_1(H_2) + e(J)}{v(J)}: J \subseteq H_1\right\},$$ setting $m_1(H_2, H_1) = 0$ if $v(H_1) = 0$.
\end{definition}
\noindent The following facts, which can be proved by elementary algebra, illustrate the relationship between 1-density and mixed 1-density. 
\begin{fact}\label{fact: mixed density}
   For all pairs of graphs $H_1$ and $H_2$ with $m_1(H_1) \geq m_1(H_2)$, 
  the mixed $1$-density of $H_2$ and $H_1$ lies between the 1-densities of the corresponding graphs:  $m_1(H_1) \geq m_1(H_2, H_1) \geq m_1(H_2)$. Furthermore, $m_1(H_1) > m_1(H_2, H_1) > m_1(H_2)$ if and only if $m_1(H_1) > m_1(H_2)$, and if for some graph $H_3$ we have $m_1(H_2)=m_1(H_3) \leq m_1(H_1)$ then $m_1(H_2, H_1)=m_1(H_3, H_1)$.
\end{fact}
\begin{theorem}[Kreuter~\cite{kreuter96}]\label{thm:kreuter}
    For $r \geq 2$ and graphs $H_1, \ldots, H_r$ such that $m_1(H_1) \geq \cdots \geq m_1(H_r)$, 
    $e(H_{1}), e(H_2) \geq 1$ and such that $H_1$ is not a matching if $r = 2$, 
    the following holds:
    there exist constants $c, C>0$ such that $$\lim_{n\to\infty} \mathbb{P}[G(n,p) \to (H_1, \ldots, H_r)_v] = \begin{cases}
    1 & \mbox{if} \ p > Cn^{-\frac{1}{m_1(H_{2}, H_1)}}    \\
    0 & \mbox{if} \ p < cn^{-\frac{1}{m_1(H_{2}, H_1)}}. 
    \end{cases}$$
\end{theorem}
\noindent In this paper, we shall obtain an analogue of Kreuter's result for \emph{randomly perturbed graphs}, which were introduced by Bohman, Frieze and Martin~\cite{bfm} in 2003. In this setting, one has a (somewhat) dense graph which is perturbed by the addition of a given number of random edges. This can be viewed as a model interpolating between the random graph $G(n,p)$ on the one hand and the (somewhat) dense or structured graphs usually studied in extremal graph theory. Randomly perturbed graphs have received extensive attention due to the often delicate interplay between extremal and probabilitic considerations required for their study, and constitute a highly active area of research in combinatorics --- there are now over 138 citations of the original paper of Bohman, Frieze and Martin --- with a particular focus on finding spanning structures in randomly perturbed graphs.

Our goal is to investigate threshold functions for (multicolour, asymmetric) vertex-Ramsey properties in randomly perturbed graphs. Given a graph $G$, we define its \emph{density} $d(G)$ to be $d(G) := e(G)/ \binom{v(G)}{2}$. 
\begin{definition}
    Let $d \in [0,1]$ and $H_1, \ldots, H_r$ be graphs. We say a function $p = p(n): \ \mathbb{N}\rightarrow [0,1]$ is a 
    \emph{perturbed vertex-Ramsey threshold function for $H_1, \ldots, H_r$ at density $d$} if:
    \begin{itemize}
        \item[(i)] there exists an infinite increasing sequence of positive integers $(n_k)_{k\in \mathbb{N}}$ and an associated sequence of $n_k$-vertex graphs $(G_{n_k})_{k \in \mathbb{N}}$, each of density at least $d$, such that if $q(n) = o(p(n))$, 
        then 
        \[\lim_{k\rightarrow \infty}\mathbb{P}\left[G_{n_k} \cup G(n_k,q(n_k)) \rightarrow (H_1, \ldots, H_r)_v\right]=0\]
        \item[(ii)] for any $q(n) = \omega(p(n))$ and any sequence $(G_n)_{n\in \mathbb{N}}$ of $n$-vertex graphs, each of density at least $d$, we have 
        \[\lim_{n\rightarrow \infty}\mathbb{P}\left[G_n \cup G(n,q(n)) \rightarrow (H_1, \ldots, H_r)_v\right]=1.\]
    \end{itemize}
\end{definition}
\noindent     We write $p(n; H_1, \ldots, H_r, d)$ as a placeholder for any instance $p(n)$ of a perturbed vertex-Ramsey threshold function for $H_1, \ldots, H_r$ at density $d$. Further if every graph on sufficiently many vertices of density at least $d$ is $(H_1, \ldots, H_r)_v$-Ramsey (and thus (i) above cannot be satisfied), we define $p(n; H_1, \ldots, H_r, d) := 0$. Hence from Observation~\ref{obs:dense}, for any $\varepsilon > 0$ we have $p(n; H_1, \ldots, H_r, d) := 0$ when $d =  (1 - 1/({\sum_{j=1}^r(\chi(H_j) - 1)}) + \varepsilon)$.
If $H_1 = \cdots = H_r = H$ for some graph $H$, we write $p(n; H, r, d)$ in place of $p(n; H_1, \ldots, H_r, d)$. Since being $(H_1, \ldots, H_r)$ vertex-Ramsey is a non-trivial monotone property, it follows from the Bollob\'as--Thomason threshold theorem~\cite{bollobas1987threshold} that perturbed vertex-Ramsey threshold functions exist for all values of $d$ for which $p(n; H_1, \ldots, H_r, d)\neq 0$. 
\begin{remark}
    In some very specific cases the vertex-Ramsey problem in the randomly perturbed setting may reduce to the same problem in the random setting. Indeed, as observed in~\cite[Observation 1.13]{dmt}, if $H_1=H_2=\ldots =H_r=H$ for some graph $H$ and $r\geq 2k:=2\lceil 1/(1-d)\rceil $, then  $p(n; H_1, \ldots, H_r, d)=m_1(H)$: for the upper bound, it follows from Theorem~\ref{thm:kreuter} that for $p\geq Cn^{-1/m_1(H)}$ we have $G(n,p)\rightarrow(H, r)_v$ with probability $1-o(1)$; for the lower bound, consider a complete balanced $k$-partite graph $G_n$ on $n$ vertices. Clearly $d(G_n)\geq d$. Now inside each of the $k$ parts $V_i$ of $G_n$, Theorem~\ref{thm:kreuter} implies that for $p\leq cn^{-1/m_1(H)}$, we have $G(n,p)[V_i]\not\rightarrow(H,H)_v$ with probability $1-o(1)$, so that by using a different pair of colours inside each of the $k$ parts we can colour the whole of $G_n\cup G(n,p)$ without creating a monochromatic copy of $H$. However in general the randomly perturbed setting is starkly different, as shown in the work of Krivelevich, Sudakov and Tetali~\cite{kst} and Das, Morris and Treglown~\cite{dmt} discussed below --- see in particular Remark~\ref{remark: density makes a difference}.
\end{remark}


Krivelevich, Sudakov and Tetali~\cite[Theorem 2.1]{kst} studied the threshold for the appearance of a fixed subgraph $H$ in the randomly perturbed graph setting, which can be viewed as the threshold for the $(K_1, H)_v$-Ramsey property (or simply the vertex-Tur\'an property for $H$). To state their result we need some definitions. First of all, we recall that the \emph{$0$-density of $H$} is $$m(H) := \max\{e(J)/v(J):\  J \subseteq H,\  v(J) > 0\}.$$ Note that this parameter governs the appearance of the graph $H$ in $G(n,p)$, see e.g.~\cite{JLR2000}. Next we define a $k$-partite version of this parameter.
\begin{definition}
Let $H$ be a graph and $k \in \mathbb{N}$. Let $\mathrm{Part}(H,k)$ denote the collection of ordered $k$-tuples of  (possibly empty) graphs $(H_1, \ldots, H_k)$ such that there exists a $k$-partition $\sqcup_{i=1}^k V_i$ of $V(H)$ with $H[V_i]=H_i$ for each $i\in [k]$.
 We define the \emph{$k$-partite $0$-density} of $H$ to be $$m(H;k) := \min_{(H_1, \ldots, H_k) \in \mathrm{Part}(H,k)} \ \max_{i:\   H_i \neq \emptyset} m(H_i).$$
\end{definition}
\begin{remark}~\label{remark: density makes a difference}
For $H=K_4$ and $k=2$, it is not hard to check that $m(K_4;2) = m(K_2) = 1/2 < 3/2 = m(K_4)$. This example together with the result below shows that the randomly perturbed setting differs from the purely random setting in general: density does make a difference.
\end{remark}
\begin{theorem}[Krivelevich, Sudakov and Tetali~\cite{kst}]\label{thm:kst}
    Let $0 < d < 1$ and set $k \geq 2$ to be the unique integer satisfying  $\frac{k-2}{k-1} < d \leq \frac{k-1}{k}$. 
    Let $H$ be a graph with $e(H) \geq 1$. Then $$p(n; K_1, H, d) = n^{-1/m(H;k)}.$$
\end{theorem}
\noindent Building on Theorem~\ref{thm:kst}, Das, Morris and Treglown initiated the study of vertex-Ramsey properties of randomly perturbed graphs, resolving the problem in the two-colour case $r=2$ when at least one of the graphs $H_1, H_2$ is a clique~\cite[Theorem 1.11]{dmt}. To state their result, it is helpful to introduce a parameter removing the need to specify $m_1(H_1)\geq m_1(H_2)$ when referring to the mixed $1$-density $m_1(H_2, H_1)$.
\begin{definition}
Given an ordered $r$-tuple of graphs $\mathbf{H}=(H_1, H_2, \ldots, H_r)$, let $\sigma$ be any permutation such that $m_1(H_{\sigma(1)}) \geq m_1(H_{\sigma(2)}) \geq \ldots \geq m_1(H_{\sigma(r)})$. We then define $$\beta (\mathbf{H})=\beta(H_1,H_2, \ldots, H_r):= m_1\left(H_{\sigma(2)}, H_{\sigma(1)}\right).$$
\noindent Further, given an $r$-tuple of families of graphs $\mathbf{\mathcal{H}}=(\mathcal{H}_1, \mathcal{H}_2, \ldots, \mathcal{H}_r)$, we define $$\beta(\mathbf{\mathcal{H}}) := \min_{H_1 \in \mathcal{H}_1,\  \ldots, \ H_r \in \mathcal{H}_r} \beta\left(H_1,H_2, \ldots, H_r\right).$$
\end{definition}
\noindent Note that $\beta(\mathbf{H})$ is well defined as, by Fact~\ref{fact: mixed density}, the value of $\beta (\mathbf{H})$ is independent of the choice of $\sigma$. Indeed, let $\sigma, \sigma'$ be permutations such that $\sigma \neq \sigma'$, $m_1(H_{\sigma(1)}) \geq m_1(H_{\sigma(2)}) \geq \ldots \geq m_1(H_{\sigma(r)})$ and $m_1(H_{\sigma'(1)}) \geq m_1(H_{\sigma'(2)}) \geq \ldots \geq m_1(H_{\sigma'(r)})$. If $m_1(H_{\sigma(1)}) = m_1(H_{\sigma(2)})$, then by Fact~\ref{fact: mixed density} we have $m_1(H_{\sigma(2)}, H_{\sigma(1)}) = m_1(H_{\sigma'(2)}, H_{\sigma'(1)}) \ (= m_1(H_{\sigma(1)}))$. Further, if $m_1(H_{\sigma(1)}) >  m_1(H_{\sigma(2)}) = m_1(H_{\sigma(3)})$, then $\sigma(1) = \sigma'(1)$ and $m_1(H_{\sigma(2)}) = m_1(H_{\sigma(3)}) = m_1(H_{\sigma'(2)}) = m_1(H_{\sigma'(3)})$. Hence by Fact~\ref{fact: mixed density} we have $m_1(H_{\sigma'(2)}, H_{\sigma'(1)}) = m_1(H_{\sigma(2)}, H_{\sigma(1)})$. 

The critical parameter governing the $(K_t, H)_v$-Ramsey property in the randomly perturbed setting is then the following.
\begin{definition}
    For $t \in \mathbb{N}$, $k \geq 2$ and a graph $H$, define $$m^*(K_t, H; k) := \max_{r_1 + \cdots + r_k \leq t-1} \ \min_{(H_1, \ldots, H_k) \in \mathrm{Part}(H,k)} \ \max_{i:\ H_i \neq \emptyset} \beta(K_{r_i + 1}, H_i),$$ 
    where the first maximum is taken over all ordered $k$-tuples $(r_1, \ldots, r_k)$
of non-negative integers summing to at most $t-1$.
\end{definition} 
\begin{theorem}[Das, Morris and Treglown~\cite{dmt}]\label{thm:dmt}
    For $0 < d < 1$, let $k \geq 2$ be the unique integer satisfying $\frac{k-2}{k-1} < d \leq \frac{k-1}{k}$.
    For any $t \in \mathbb{N}$ and graph $H$,
    we have $$p(n; K_t, H, d) = n^{-1/m^*(K_t, H; k)}.$$ 
\end{theorem} 
\noindent Note that since $\beta (K_1, H)=m(H)$ for any graph $H$, the case $t=1$ of Theorem~\ref{thm:dmt} coincides precisely with Theorem~\ref{thm:kst}, as we should expect.

\subsection{Contributions of this paper}
In their concluding remarks, Das, Morris and Treglown~\cite[Section 4]{dmt} wrote that  ``the most pressing problem that remains open is to extend our results to all pairs of graphs $(F, H)$ with the symmetric case $F = H$ of particular interest''. Our contribution in this paper is to determine a threshold for the multicolour asymmetric vertex-Ramsey property in randomly perturbed dense graphs, resolving this problem in full generality. To state our results, we must make some definitions. 
\begin{definition}
Let $H$ be a graph, and let $k\geq 2$. We define a \emph{$k$-partite certificate of $H$-freeness} $\mathcal{F}(H;k)=(\mathcal{F}_1, \mathcal{F}_2, \ldots, \mathcal{F}_k)$ to be an (ordered) $k$-tuple of collections of subgraphs of $H$
such that for each $k$-tuple $(H_1,  \ldots, H_k) \in \mathrm{Part}(H,k)$, 
there is some $j$ such that $H_j\in \mathcal{F}_j$. 
\end{definition}
\noindent We note that a similar notion appears in the concluding remarks of~\cite{dmt} under the name `$k$-cover of $H$'; see also the lower bound on the threshold for the asymmetric vertex-Ramsey property given in~\cite[Inequality (4.3)]{dmt}. 
Next, given an ordered $r$-tuple of graphs $\mathbf{H}=(H_1, \ldots, H_r)$ for some $r\geq 2$, and an integer $k\geq 2$, we set
\begin{align}\label{eq: critical threshold}
\beta(\mathbf{H}; k):= \underset{\mathcal{F}(H_1;k),\ldots, \mathcal{F}(H_r;k)}{\max}\ \underset{i\in [k]}{\min}\  \beta\left(\mathcal{F}(H_1;k)_i,\ldots,  \mathcal{F}(H_r;k)_i\right),
\end{align}
where the maximum is taken over all ordered $r$-tuples $(\mathcal{F}(H_1;k), \ldots , \mathcal{F}(H_r;k))$ of $k$-partite certificates of $H_1$-, $H_2$-, $\ldots$, $H_r$-freeness.
\begin{remark}
The quantity $\beta$ is computable in principle. Indeed, 
there are at most $2^{k\cdot{2}^ {\vert V(H)\vert}}$ many $k$-partite certificates of $H$-freeness so the maximum in~\eqref{eq: critical threshold} is taken over a finite set, and the Kreuter threshold inherent in the parameter $\beta\big(\mathcal{F}(H_1;k)_i, \ldots,  \mathcal{F}(H_r;k)_i\big)$ is itself a maximum over finitely many pairs of finite graphs of a maximum over their finitely many subgraphs. However in practice computing $\beta(\mathbf{H}; k)$ is not a tractable task unless the graphs $H_1, \ldots, H_r$ are of small order 
or have a very simple structure, such as being cliques. 
\end{remark}
\begin{remark}
While there are superficial differences in the way we have presented these parameters, it is easily checked that $m^{*}(K_t, H; k)=\beta((K_t, H); k)$, and the reader is encouraged to verify this in order to check their understanding of the crucial definition~\eqref{eq: critical threshold}.
\end{remark}
\begin{theorem}\label{theorem: main}
Let $\mathbf{H}=(H_1, H_2, \ldots H_r)$ be an $r$-tuple of graphs for some $r\geq 2$. Let $d\in (0,1)$ be fixed with $\frac{k-2}{k-1} < d \leq \frac{k-1}{k}$. Then 
\begin{align*}
p(n; \mathbf{H}, d)= n^{-1/\beta(\mathbf{H}; k)}. 
\end{align*}
\end{theorem}
\noindent Before proving Theorem~\ref{theorem: main} in full generality, we shall prove it in the special case when the randomly perturbed graph is a Tur\'an graph:
\begin{theorem}\label{theorem: turan}
Let $T_{k}(n)$ denote the complete balanced $k$-partite graph on $n$ vertices. Consider $G=T_{k}(n)\cup G_{n,p}$. Then for any $r\geq 2$ fixed and any $r$-tuple of graphs $\mathbf{H}=(H_1, H_2, \ldots, H_r)$ there exists a constant $C>0$ such that the following holds:
\begin{enumerate}[(i)]
    \item if $p =o(n^{-1/\beta(\mathbf{H}; k)}) $ then with probability $1-o(1)$ there exists an $r$-colouring of the vertices of $G$ such that for every $j\in [r]$ there is no monochromatic copy of $H_j$ in colour $j$;
    \item if $p>C n^{-1/\beta(\mathbf{H}; k)} $, then with probability $1-o(1)$ in every $r$-colouring of the vertices of $G$, there is some $j\in [r]$ such that $G$ contains a monochromatic copy of $F_j$ in colour $j$.
\end{enumerate}
\end{theorem}
\noindent Theorem~\ref{theorem: turan} has the double virtue of having a rather simple and transparent proof, 
giving the intuition as to why the claimed threshold in Theorem~\ref{theorem: main} is the correct one, 
and of implying the $0$-statement of Theorem~\ref{theorem: main}. We prove Theorem~\ref{theorem: turan} in Section~\ref{section: turan},
before moving on to the more technical proof of the $1$-statement of Theorem~\ref{theorem: main} in Section~\ref{section: 1-statement}, 
which relies on regularity methods and robust versions of Theorem~\ref{thm:kreuter} similar to the ones employed by Das, Morris and Treglown.

\subsection{Notation}

Throughout we omit floors and ceilings where they are not pertinent to the arguments. For a graph $G$, we write $v(G)$ and $e(G)$ for the number of vertices and of edges in $G$, respectively. We say a graph is non-empty if it has at least one vertex. Given a set of vertices $V'$ in $V(G)$ write $G[V']$ for the subgraph of $G$ induced by $V'$. We write $H\subseteq G$ if $H$ is a subgraph of $G$, and, for graphs $G$ and $G'$ on the same vertex set we write $G\cup G'$ for their union. Given a set $X$ and $t \in \mathbb{N}$, we write $\binom{X}{t}$ for the set of all subsets of $X$ of size $t$. If $X$ is a set of vertices and $F$ is a graph, we write $\binom{X}{F}$ for the set of all possible copies of $F$ supported on vertices in $X$; we consider two possible copies of $F$ to be different if they have distinct sets of edges. We also write $\binom{G}{F}$ for the set of copies of $F$ in a graph $G$ (so that $\binom{X}{F}$ defined previously is precisely $\binom{G}{F}$ for $G$ the complete graph on the vertex set $X$).
Throughout the paper, we use standard Landau notation. 
\section{Tools and useful results}
We will need the following generalisation of the 0-statement of Theorem~\ref{thm:kreuter}. 
\begin{proposition}\label{prop:0family}
    Let $(\mathcal{H}_1, \ldots, \mathcal{H}_r)$ be an $r$-tuple of finite families of non-empty graphs. If $p = o(n^{-1/\beta(\mathcal{H}_1, \ldots, \mathcal{H}_r)})$, then with probability $1-o(1)$ there is an $r$-colouring of the vertices of $G(n,p)$ 
    without a monochromatic copy of any $H_j \in \mathcal{H}_j$ in colour $j$ for any $j\in [r]$.
\end{proposition}
\noindent Das, Morris and Treglown stated the $r=2$ case of Proposition~\ref{prop:0family} in~\cite[Proposition 2.1]{dmt}, commenting that `the proof of Proposition~\ref{prop:0family} is nearly identical to the proof of Theorem~\ref{thm:kreuter}\ldots' and providing a sketch proof. The extension of ~\cite[Proposition 2.1]{dmt} to the general $r$ case is trivial (since in fact $\beta(\mathcal{H}_1,  \ldots, \mathcal{H}_r)=\max_{i_1<i_2}\beta(\mathcal{H}_{i_1},  \mathcal{H}_{i_2})$  and it thus suffices to find a colouring using only some pair of colours $j_1$ and $j_2$ without any monochromatic copy of $H_{j_1}\in \mathcal{H}_{j_1}$ in colour $j_1$ or monochromatic copy of  $H_{j_2}\in \mathcal{H}_{j_2}$ in colour $j_2$). Proposition~\ref{prop:0family} is stated with $p = o(n^{-1/\beta(\mathcal{H}_1,\ldots,\ \mathcal{H}_r)})$ instead of
$p \leq cn^{-1/\beta(\mathcal{H}_1,\ldots,\ \mathcal{H}_r)}$
for some sufficiently small constant $c$ so that if $\mathcal{H}_1$ and $\mathcal{H}_2$ contain matchings 
the conclusion still holds (see the discussion after Theorem~\ref{thm:lrv}).

While in Theorem~\ref{theorem: turan} we work with a complete $k$-partite graph $G_n = T_k(n)$, in the proof of Theorem~\ref{theorem: main} we will have to handle the case where $G_n$ is any graph of density $d > 1 - 1/(k-1)$. We will use the seminal Szemer\'{e}di Regularity Lemma~\cite{szemeredi} to find a large subgraph in $G_n$ with properties approximating those of a complete $k$-partite graph. Let us collect the relevant definitions and tools. We refer a reader to~\cite{komlos1995szemeredi} as a reference on the regularity lemma.
\begin{definition}
    Given $\varepsilon > 0$, 
    a graph $G$ and vertex disjoint sets $A, B \subseteq V(G)$, 
    we say the pair $(A, B)$ is $\varepsilon$-regular if for every $X \subseteq A$ and $Y \subseteq B$ with $|X| > \varepsilon |A|$ and $|Y  > \varepsilon |B|$, 
    we have $|d(X,Y) - d(A,B)| < \varepsilon$, 
    where $d(S,T) := e(S,T)/(|S||T|)$ for all $S, T \subseteq V(G)$.
\end{definition}
\noindent As is well known, the random-like distribution of edges in regular pairs allows them to mimic complete bipartite graphs with respect to containing (sparse) bipartite subgraphs. 
\begin{lemma}\label{lem:focusandcut}
    Let $\varepsilon > 0$. For any $\varepsilon$-regular pair $(A,B)$ in $G$ with $d(A,B) = d$:

    \begin{itemize}
        \item[(i)] If $\ell \geq 1$ and $(d - \varepsilon)^{\ell - 1} > \varepsilon$, then $$|\{(x_1, \ldots, x_{\ell}) \in A^{\ell}: |\cap_i N(x_i) \cap B| \leq (d - \varepsilon)^{\ell}|B|\}| \leq \ell \varepsilon |A|^{\ell}.$$
        \item[(ii)] If $\gamma > \varepsilon$, and $A' \subset A$ and $B' \subset B$ are such that $|A'| \geq \gamma |A|$ and $|B'| \geq \gamma |B|$,
        then $(A', B')$ is an $\varepsilon'$-regular pair of density $d'$, where $\varepsilon' := \max\{\varepsilon/\gamma, 2\varepsilon\}$ and $|d' - d| < \varepsilon$.
    \end{itemize}
\end{lemma}
\noindent Property (i) essentially says that small sets of vertices typically have many common neighbours in regular pairs. 
Property (ii) relates how sufficiently large subgraphs of dense regular pairs remain both regular and dense. In our arguments, we will use the following simple corollary of Szemer\'{e}di's regularity lemma~\cite{szemeredi}
and Tur\'{a}n's theorem~\cite{turan}.
\begin{proposition}[Szemer\'edi~\cite{szemeredi}]\label{prop:sze}
    For every $k \geq 2$ and $\alpha, \varepsilon > 0$ with $\alpha \geq 6\varepsilon$, 
    there exist constants $\eta := \eta(k, \alpha, \varepsilon) > 0$ and $n_0 := n_0(k, \alpha, \varepsilon)$ such that,
    for all $n \geq n_0$ and graphs $G$ on $n$ vertices with $d(G) \geq 1 - 1/(k-1) + 2 \alpha$, the following holds: 
    there are pairwise disjoint vertex sets $V_1, \ldots, V_k \subseteq V(G)$ with $|V_1| = \cdots = |V_k| \geq \eta n$ such that, 
    for each $1\leq i< i' \leq k$, the pair $(V_{i}, V_{i'})$ is $\varepsilon$-regular and $d(V_i, V_{i'}) \geq \alpha$.
\end{proposition}

Our main probabilistic tool is the following `robust' version of the 1-statement of Theorem~\ref{thm:kreuter}, proved in the $r=2$ case by Das, Morris and Treglown~\cite[Theorem 2.8]{dmt}. It allows us to place down different subgraphs in different parts of the Szemer\'{e}di partition guaranteed by Proposition~\ref{prop:sze} that will interact effectively with the underlying dense graph.
\begin{theorem}\label{thm:robust1statement} Let $\mathbf{H}=(H_1, H_2,\ldots, H_r)$ be an $r$-tuple of graphs. Then there exist $\delta_0$, $c_0 > 0$ such that for all constants $\delta$ with $0 < \delta < \delta_0$, all $t = t(n) \leq \exp(n^{c_0})$ and all constants $\eta_0 > 0$, 
there exists a constant $C > 0$ such that the following holds: 
let  $\mathcal{U} = \{(U_{t'}, \mathcal{H}_{1,t'}, \mathcal{H}_{2,t'}, \ldots ,\mathcal{H}_{r,t'}):\  t' \in [t]\}$ be a collection of $(r+1)$-tuples such that for each $t' \in [t]$ and $j\in [r]$ we have \begin{itemize}
    \item $U_{t'} \subseteq [n]$ with $|U_{t'}| \geq \eta_0 n$, 
    \item $\mathcal{H}_{j,t'} \subseteq \binom{U_{t'}}{H_{j}}$ with $|\binom{U_{t'}}{H_{j}}\setminus \mathcal{H}_{j,t'}| \leq \delta |U_{t'}|^{v(H_j)}$.
\end{itemize} Then if $p \geq Cn^{-1/\beta(\mathbf{H})}$, with probability $1-o(1)$ we have that in every $r$-colouring of the vertex-set $[n]$ of $G(n,p)$ and every $t' \in [t]$ there is some colour $j=j(t')\in [r]$ and a monochromatic set $X\in \mathcal{H}_{j,t'}$ in colour $j$ such that $H_j\subseteq G(n,p)[X]$.  
\end{theorem}
\begin{proof}
The proof of Theorem~\ref{thm:robust1statement} for all $r \geq 2$ case is obtained by combining the calculations from Das, Morris and Treglown's proof in the $r=2$ case~\cite[Theorem 2.8]{dmt} with the ideas from Kreuter's original proof of Theorem~\ref{thm:kreuter}. Assume without loss of generality that $m_1(H_1)\geq m_1(H_2)\geq \ldots \geq m(H_r)$. We shall reveal $G(n,p)$ in $r$ rounds of exposure as the union of $r $ random graphs $G_1\cup G_2\cup \ldots \cup G_r$ with $G_j\sim G(n, p')$ for every $j\in [r]$, where $p'= C'n^{-1/\beta(\mathbf{H})}$  for some suitably large constant $C'$.

Let $c_1:=m_1(H_2)/m_1(H_2, H_1)$ and $\ell:=\min \left\{v(J)- (e(J)/m_1(H_2, H_1)): \ J\subseteq H_1, \ e(J)>0\right\}$. As shown in~\cite[Inequalities (3.2) and (3.3)]{dmt}, we have $c_1\leq \ell \leq 1$. Further, Das, Morris and Treglown showed in~\cite[Proof of Theorem 2.8, p.\ 1001]{dmt} that there exist constants $c_0,\gamma_0>0$ depending only on $H_1$, $H_2$ and $\eta_0$ such that if $t\leq \exp(n^{c_0})$ then with probability $1-o(1)$, for every $t'\in [t]$, $G_1[U_{t'}]$ contains a collection $\mathcal{D}_{t'} \subseteq \mathcal{H}_{1,t'}$ of $(r-1)\gamma_0 n^{\ell} $ vertex-disjoint $v(H_1)$-sized sets $X$ with $H_1\subseteq G_1[X]$.

Given such a collection $\mathcal{D}_{t'}$, let $\mathcal{T}(\mathcal{D}_{t'}):=\{T\subseteq U_{t'}: \ \vert T\cap X\vert =1 \ \forall X\in \mathcal{D}_{t'}\}$ denote the collection of transversals of $\mathcal{D}_{t'}$. Further, let $\mathcal{T}'(\mathcal{D}_{t'})=\{T' \in \binom{T}{\gamma_0n^{\ell}}: \ T\in \mathcal{T}(\mathcal{D}_{t'})\}$ denote the collection of subsets of transversals of $\mathcal{D}_{t'}$ of size precisely $\tilde{n}=\gamma_0n^{\ell}$. Clearly, we have $$\vert \mathcal{T}'(\mathcal{D}_{t'})\vert=v(H_1)^{(r-1)\gamma_0n^{\ell}} \binom{(r-1)\gamma_0n^{\ell}}{\gamma_0 n^{\ell}}\leq \exp\left(C''n^{\ell}\right)$$ for some constant $C''$ depending only on $H_1, H_2$ and $\eta_0$.

Now for any $j\in [r]\setminus[1]$, as shown in~\cite[Proof of Theorem 2.8, p.\ 1002]{dmt}, choosing $C'$ sufficiently large ensures that with probability $1-\exp\left(-\Omega(n^{\ell})\right)$, for every set $T'\in \mathcal{T}'(\mathcal{D}_{t'})$ the graph $G_j[T']$ contains at least $\tilde{n}^{v(H_j)} (p')^{e(H_j)}/2$ copies of $H_j$. Further as shown in~\cite[Proof of Theorem 2.8, p.\ 1003]{dmt}, with probability $1-\exp\left(-\Omega(n^{\ell/e(H_j)})\right)$, at most $\tilde{n}^{v(H_j)}(p')^{e(H_j)}/4$ of these copies fail to lie in $\mathcal{H}_{j, t'}$. Taking a union bound over all $j\in [r]\setminus [1]$, and picking $c_0<\min\left\{c_1, \ \min \left\{\frac{\ell}{e(H_j)}: \ 2\leq j\leq r\right\}\right\}$, we conclude that with probability $1-o(1)$ the following holds for every $t'\in [t]$:
\begin{enumerate}[(a)]
    \item there is a set $\mathcal{D}_{t'}$ of  $(r-1)\tilde{n}$ vertex-disjoint sets $X\in \mathcal{H}_{1,t'}$ such that $H_1\subseteq G_1[X]$;
    \item for every $j\in [r]\setminus[1]$ and every $T'\in \mathcal{T}'(\mathcal{D}_{t'})$, there exists a set $X\in \binom{T'}{ v(H_j)}\cap \mathcal{H}_{j,t'}$ such that $H_j\subseteq G_j[X]$.
\end{enumerate}
Let us condition on this event holding. Then for every colouring of $[n]$ and every $t'\in [t]$, we have by (a) above that either there is a monochromatic set $X\in \mathcal{D}_{t'}\subseteq \mathcal{H}_{1,t'}$ in colour $1$ such that $H_1\subseteq G_1[X]\subseteq G(n,p)[X]$, or there is a transversal $T \in \mathcal{T}(\mathcal{D}_{t'})$ on which the colour $1$ does not appear. In the latter case, by averaging there is some $T' \subseteq T$ in $\mathcal{T}'(\mathcal{D}_{t'})$ such that $T'$ is monochromatic in colour $j$ for some $j\in [r]/[1]$. But then by (b) above this yields a subset $X\in\mathcal{H}_{j,t'}$ monochromatic in colour $j$ such that $H_j\subseteq G_j[X]\subseteq G(n,p)[X]$, and we are done.
\end{proof}

\section{The threshold in Tur\'an graphs}\label{section: turan}
Let us prove Theorem~\ref{theorem: turan}.
\begin{proof}[Proof of Theorem~\ref{theorem: turan}(i)]
    Fix certificates $\mathcal{F}(H_1;k), \ldots,  \mathcal{F}(H_r;k)$ that maximise $\beta(\mathbf{H}; k)$.  Label the $k$ parts of $T_k(n)$ as $V_1, \ldots, V_k$. Since $\beta(\mathbf{H}; k) \leq \beta(\mathcal{F}(H_1; k)_i,\  \mathcal{F}(H_2;k)_i,\  \ldots, \mathcal{F}(H_r; k)_i)$ for all $i \in [k]$, we may apply Proposition~\ref{prop:0family} for each $i \in [k]$ 
    to yield with probability $1-o(1)$ an $r$-colouring $c_i$ of $V_i$ such that $T_k(n)\cup G_{n,p}[V_i]$ contains no monochromatic copy of a graph in $\mathcal{F}(H_j:k)_i$ in colour $j$ for any $j\in [r]$. We then let $c$ denote the union of the $k$ colourings $c_i$, $i\in [k]$.

    We claim that for every $j\in [r]$, there is no monochromatic copy of $H_j$ in colour $j$ in the vertex-colouring $c$ of $T_k(n)\cup G_{n,p}$. Indeed, every copy of $H_j$ in $T_k(n)\cup G_{n,p}$ yields an ordered $k$-tuple $((H_j)_1, \ldots, (H_j)_k)\in \mathrm{Part}(H, k)$, with $(H_j)_i$ the subgraph of our copy of $H_j$ induced by the vertices in $V_i$. Since $\mathcal{F}(H_{j};k)$ is a $k$-partite certificate of $H_{j}$-freeness, for each such $k$-tuple there is some index $i_0\in [k]$ such that $(H_j)_{i_0} \in \mathcal{F}(H_{j};k)_{i_0}$. By construction of $c_{i_0}$, there is no monochromatic copy of $(H_j)_{i_0}$ in $T_k(n)\cup G_{n,p}[V_{i_0}]$ in colour $j$. This in turn implies that there is no monochromatic copy of $H_j$ in colour $j$ in the $r$-colouring $c$ of $T_{k}(n)\cup G_{n,p}$.
\end{proof}
\begin{proof}[Proof of Theorem~\ref{theorem: turan}(ii)]
    Let $G=T_k(n)\cup G_{n,p}$, for $p>Cn^{-1/\beta(\mathbf{H};k)}$ and some constant $C>0$ to be specified later. Let $V_1,\ \ldots,\  V_k$ denote the $k$ parts of our $T_k(n)$. By definition, for any ordered $k$-tuples of $k$-partite certificates of $H_1$-,  $\ldots$, $H_r$-freeneess $\left(\mathcal{F}(H_1;k), \ldots, \mathcal{F}(H_r;k)\right)$ we have $\beta(\mathbf{H};k) \geq\min_{i\in [k]}\beta (\mathcal{F}(H_1; k)_i, \ldots, \mathcal{F}(H_r;k)_i)$.

    By Theorem~\ref{thm:kreuter} applied to the (finite) set of ordered $k$-tuples of $k$-partite certificates of $H_1$-, $\ldots$, $H_r$-freeness $\left(\mathcal{F}(H_1;k), \ldots, \mathcal{F}(H_r;k)\right)$ and each $i_0\in [k]$, there exists a choice of $C$ ensuring that with probability $1-o(1)$ for every such $\left(\mathcal{F}(H_1;k), \ldots, \mathcal{F}(H_j;k)\right)$ and every $i_0\in [k]$ such that $\min_{i\in [k]}\beta (\mathcal{F}(H_1; k)_i, \ldots, \mathcal{F}(H_r;k)_i)$ is achieved at $i=i_0$, we have $G[V_{i_0}]=G_{n,p}[V_{i_0}]\rightarrow (\mathcal{F}(H_1;k)_{i_0}, \ldots ,\mathcal{F}(H_r;k)_{i_0})_v$. Let $\mathcal{E}$ denote this event.

    We claim that if $\mathcal{E}$ occurs, then $G\rightarrow (H_1, \ldots, H_r)_v$. Indeed, assume for a contradiction that $\mathcal{E}$ occurs but there exists an $r$-colouring $c$ of $V(G)$ such that for every $j\in [r]$ there is no monochromatic copy of $H_{j}$ in colour $j$ in $G$. Then for every $j\in [r]$ and every $((H_{j})_1, \ldots, (H_{j})_k) \in \mathrm{Part}(H_{j}, k)$, there must be some $i \in [k]$ such that $G[V_{i}]$ contains no monochromatic copy of $(H_j)_{i}$ in colour $j$ (otherwise we would have a monochromatic copy of $H$ in colour $j$ inside $G$). We can thus construct a $k$-partite certificate of $H_j$-freeness $\mathcal{F}(H_j; k)$ by letting $\mathcal{F}(H_j; k)_{i}$ be the collection of all $(H_j)_i$ such that there exists a $k$-tuple $((H_{j})_1, \ldots, (H_{j})_k) \in \mathrm{Part}(H_{j}, k)$ such that $G[V_i]$ contains no monochromatic copy of $(H_j)_i$ in colour $j$.
    We thus obtain an $r$-tuple $(\mathcal{F}(H_1; k), \ldots, \mathcal{F}(H_r;k))$, where for each $j\in [r]$ the $k$-tuple $\mathcal{F}(H_j;k)$ is a $k$-partite certificate of $H_j$-freeness. Let $i_0\in [k]$ be any index such that $\min_{i\in [k]}\beta (\mathcal{F}(H_1; k)_i, \ldots, \mathcal{F}(H_r;k)_i)$ is achieved at $i=i_0$. Then by construction of $c$ we have that $G[V_{i_0}]=G_{n,p}[V_{i_0}]\not\rightarrow (\mathcal{F}(H_1;k)_{i_0}, \ldots \mathcal{F}(H_r;k)_{i_0})_v$, contradicting our assumption that $\mathcal{E}$ holds.\end{proof}



\section{The threshold in general graphs}\label{section: 1-statement}
\noindent As noted earlier, Theorem~\ref{theorem: turan}(i) immediately implies the 0-statement of Theorem~\ref{theorem: main}. All that remains therefore is to prove the 1-statement of Theorem~\ref{theorem: main}.
\begin{proof}[Proof of the 1-statement of Theorem~\ref{theorem: main}]
    Take $\alpha > 0$ such that $d \geq 1 - 1/(k-1) + 2\alpha$. Let $\varepsilon>0$ be a sufficiently small constant to be specified later. Let $\eta>0$ and $n_0\in \mathbb{Z}_{>0}$ be the constants whose existence is guaranteed by Proposition~\ref{prop:sze}. For $n\geq n_0$, consider an $n$-vertex graph $G_n$ with $d(G_n)\geq d$, and let $c$ be an arbitrary $r$-colouring of $G_n \cup G(n,p)$.

    By Proposition~\ref{prop:sze}, there exist $k$ pairwise disjoint sets $V_1, V_2, \ldots, V_k$ in $V(G_n)$, each of size at least $\eta n$, such that each pair $(V_i, V_{i'})$ is $\varepsilon$-regular in $G_n$ with density at least $\alpha$. We shall follow an algorithmic procedure that will involve repeatedly passing to subsets inside the sets $V_i$ in order to find either a monochromatic copy of $H_j$ for some $j\in [r]$ or a $k$-partite certificate of $H_j$-freeness for every $j\in [r]$ - we will show the latter holds with probability at most $o(1)$.

    Initially set $U_i=V_i$ for each $i\in [k]$. Throughout the proof, we say a set of vertices $X\subseteq U_i$ is
    \textbf{popular} if the vertices 
    in $X$ have at least $(\frac{\alpha}{2})^{v(X)}\vert U_{i'}\vert $ common $G_n$-neighbours in $U_{i'}$ for every $i' > i$. We now describe our algorithmic procedure.



For each colour $j\in [r]$, we initially set $\mathcal{F}(H_j;k)$ to be a $k$-tuple of empty sets. For each $k$-tuple $((H_{j})_1, \ldots, (H_{j})_k) \in \mathrm{Part}(H_{j}, k)$, we then sweep through the sets $U_i$, $1\leq i\leq k$, in order, as follows. In step $i$, if $(H_j)_i$ is an empty graph, we move to step $i+1$. Otherwise, we check whether or not there is a $v((H_j)_i)$-sized set $X_i\subseteq U_i$ such that
\begin{enumerate}[(a)]
    \item $X_i$ is monochromatic in colour $j$,
    \item $X_i$ is popular, and 
    \item $(H_j)_i$ is a subgraph of $G(n,p)[X_i]$.
\end{enumerate} 
If such a set $X_i$ exists, then for every $i'>i$ we replace $U_{i'}$ by the common neighbourhood of $X_i$ in $U_{i'}$ and then move to step $i+1$. Otherwise, if no such set $X_i$ exists, we add the graph $(H_j)_i$ to $\mathcal{F}(H_j;k)_i$ and move to step $i+1$. When we reach step $k+1$, we then move on to the next $k$-tuple from $\mathrm{Part}(H_{j}, k)$.

Observe that when we reach step $k+1$, one of two things must have happened: either we have a collection of sets $X_i\subseteq U_i$, $i\in [k]$, each monochromatic in colour $j$, such that $(H_j)_i\subseteq G(n,p)[X_i]$ and for every $i<i'$ the graph $G_n$ contains the complete bipartite graph between $X_i$ and $X_{i'}$ - in which case  we have clearly found a copy of $H_j$ in colour $j$ in $G_n\cup G(n,p)$ - or we have added $(H_j)_i$ to $\mathcal{F}(H_j; k)_i$ for some $i\in [k]$.

Thus when we have completed our algorithmic procedure for every colour $j\in [r]$, we must either have found a monochromatic copy of $H_j$ in colour $j$ for some $j\in [r]$ or produced for every $j\in [r]$ a $k$-partite certificate of $H_j$-freeness $\mathcal{F}(H_j;k)$. Let us show the latter is a $o(1)$ probability event. After our algorithmic procedure has run its course, let $(\widetilde{U}_{1}, \ldots , \widetilde{U}_r)$ denote the final state of the sequence of sets $(U_{1}, \ldots, U_r)$ found during the procedure.

Set $s:=k\sum_{j=1}^r \vert \mathrm{Part}(H_j;k)\vert$ and $v:=\max_{1\leq j\leq r}v(H_j)$. Note that $s$ is an upper bound on the number of times we shrink a set $U_{i'}$ by replacing it with the common neighbourhood inside $U_{i'}$ of some set $X_i$ (which will always have size at most $v$), and further that $s,v$ are constants depending only on $H_1, \ldots, H_r, k$. Setting $\gamma:= \left(\frac{\alpha}{2}\right)^{sv}$, we may choose $\varepsilon > 0$ sufficiently small such that the following inequalities are satisfied for a sufficiently small choice of $\delta_0$ that will be made later:
\begin{align}\label{eq: epsilon bounds}
  \gamma >\varepsilon,  && \left(\alpha-\varepsilon- \frac{\varepsilon}{\gamma}\right)^v>\left(\frac{\alpha}{2}\right)^v > \frac{\varepsilon}{\gamma}, \textrm{ and }&& \frac{\delta_0}{2}>\frac{sv\varepsilon}{\gamma}.
\end{align}
\noindent Now, throughout our procedure, for every $i\in [k]$ the set $U_i$ satisfies
\begin{align}\label{eq: lower bound on the Ui}
\vert U_i\vert \geq \vert \widetilde{U}_i\vert\geq   \left(\frac{\alpha}{2}\right)^{sv}\vert V_i\vert= \gamma\vert V_i\vert \geq \gamma \eta n=:\eta_0 n.
\end{align} In particular, it follows from Lemma~\ref{lem:focusandcut} (ii) and the first inequality in~\eqref{eq: epsilon bounds} that for every $1\leq i< i'\leq k$ the pair $(\widetilde{U}_i, U_{i'})$ is an $\varepsilon'$-regular pair of density $\alpha'$, where $\varepsilon'=\varepsilon/\gamma$ and $\alpha'\geq \alpha-\varepsilon$. Applying Lemma~\ref{lem:focusandcut} (i) and the second inequality in~\eqref{eq: epsilon bounds}, we see that at every stage of the algorithm, for every $v'\in[v]$ at most 
     $v' \varepsilon'\vert \widetilde{U}_i\vert^{v'}$
sets $X_i\subseteq \widetilde{U}_i$ of size $v'$ fail to be popular (with respect to the current values of $U_{i'}:\ i'>i$). In particular, using the third inequality in~\eqref{eq: epsilon bounds}, the number of $v'$-sized sets in $\widetilde{U_i}$ that fail to be popular at some stage in the algorithm is at most
\begin{align}\label{eq: delta0 bound}
    sv' \varepsilon'\vert \widetilde{U}_i\vert^{v'}  <\frac{
\delta_0}{2} \vert \widetilde{U}_i\vert^{v'}.
\end{align}
\noindent Further, the evolution of our sequence of sets $(U_1, \ldots, U_k)$ until it reaches its final state $(\widetilde{U}_1,\ldots ,\widetilde{U}_k)$ is determined by the choice of the original sets $(V_1, \ldots, V_k)$ and the successive choices of the sets $X_i$, whose union contains a total of at most $sv$ vertices. The number $t=t(n)$ of possible such sequences is thus at most $n^{sv}$. We arbitrarily order these sequences, and let $(\widetilde{U}_{1,t'}, \widetilde{U}_{2,t'}, \ldots, \widetilde{U}_{k,t'})$ denote for each $t'\in [t]$ the final state of $(U_1, \ldots, U_k)$ in the $t'$-th sequence.

By definition of $\beta$, there exists some $i_0 \in [k]$ and for each $j\in [r]$ a graph $(H_j)_{i_0}\in \mathcal{F}(H_j; k)_{i_0}$ such that  $\beta(\mathbf{H};k)\geq \beta ((H_1)_{i_0}, \ldots, (H_r)_{i_0})$. For every $t'\in [t]$ and every $j\in [r]$, we let $\mathcal{H}_{j,t'}$ denote the collection of $v((H_j)_{i_0})$-sized subsets of $\widetilde{U}_{i_0,t'}$ that are popular at every stage of the evolution of the $t'$-th sequence of sets $(U_1, \ldots, U_k)$. We then set $\mathcal{U}:=\{(\widetilde{U}_{i_0,t'}, \mathcal{H}_{1, t'}, \ldots , \mathcal{H}_{r,t'}): \ t'\in [t]\}$.

We let $\delta_0,c_0>0$ be the constants whose existence is guaranteed by Theorem~\ref{thm:robust1statement} applied to the $r$-tuple $((H_1)_{i_0}, \ldots, (H_r)_{i_0})$. Setting $\delta=\delta_0/2$, recalling we have defined $\eta_0=\gamma \eta$ above, and using the inequalities~\eqref{eq: lower bound on the Ui} and~\eqref{eq: delta0 bound}, we see the parameter $t=t(n)=o(e^{n^c})$ and the family $\mathcal{U}$ defined above satisfy the assumptions required to apply Theorem~\ref{thm:robust1statement}: there exists $C>0$ such that for $p \geq Cn^{-1/\beta((H_1)_{i_0}, \ldots , (H_r)_{i_0})}$, with probability $1-o(1)$ 
there is for every $t'\in [t]$ some colour $j$ such that there is a monochromatic copy of $(H_j)_{i_0}$ in colour $j$ based on some ever-popular set in $\mathcal{H}_{j,t'}$.

If these monochromatic copies could be found, then regardless of which of the $t$ possible sequences of the $k$-tuple $(U_1, \ldots, U_k)$ our algorithmic procedure gave rise to, one of the $(H_j)_{i_0}$ would not have been included in $\mathcal{F}(H_j)_{i_0}$. In other words, for $p\geq Cn^{-1/\beta(\mathbf{H};k)}\geq Cn^{-1/\beta((H_1)_{i_0}, \ldots , (H_r)_{i_0})}$, the probability that our algorithmic procedure does not find a monochromatic copy of $H_j$ in colour $j$ for some $j\in [r]$ in $G_n\cup G(n,p)$ but instead outputs a $k$-partite certificate of $H_j$-freeness for every $j\in [r]$ is $o(1)$, proving the theorem.
\end{proof}

\section{Concluding remarks}

In this paper we have focused on vertex-Ramsey properties of randomly perturbed graphs. 
The study of edge-Ramsey properties of randomly perturbed graphs is ongoing. 
For graphs $H_1$, $H_2$, we say a graph $G$ is $(H_1, H_2)_e$-Ramsey if every red/blue edge colouring of the edges of $G$ yields a monochromatic copy of $H_1$ in red or a monochromatic copy of $H_2$ in blue.
The study of edge-Ramsey properties of randomly perturbed graphs was initiated by Krivelevich, Sudakov and Tetali~\cite{kst} who determined for all $t \geq 3$ the threshold at which adding a copy of $G(n,p)$ to a dense graph results in a graph that is $(K_3, K_t)_e$-Ramsey.
This was followed by work of Das and Treglown~\cite{dt} who almost completely answered a question of
Krivelevich, Sudakov and Tetali~\cite{kst} by generalising their result to cover all pairs of cliques except $(K_4, K_t)$, for $t \geq 5$, 
with the result for $(K_4, K_4)$ following from bounds given by Das and Treglown~\cite{dt} and a construction of Powierski~\cite{powierski}.
Das and Treglown~\cite{dt} also addressed the corresponding problems for pairs of cycles fully and for pairs of cycles and cliques in almost all cases.

All these results were obtained before the recent resolution of the Kohayakawa-Kreuter~\cite{kk} conjecture on the edge-Ramsey properties of $G(n,p)$ by Mousset, Nenadov and Samotij~\cite{mns} (for the 1-statement),
Kuperwasser, Widgerson and Samotij~\cite{ksw},
Bowtell, Hancock~\cite{bhh} and the third author, and Christoph, Martinsson, Steiner and Wigderson~\cite{cmsw} (for the 0-statement).
 The confirmation of this conjecture paves the way for building on the results of~\cite{dt, kst, powierski} and determining the threshold for edge-Ramsey properties in randomly perturbed random graphs in full generality.
The very recent work of Yancey~\cite{y} on a generalisation of the Kohayakawa-Kreuter conjecture to families of graphs via a graph partitioning conjecture of Kuperwasser, Widgerson and Samotij~\cite{ksw} may also prove useful to this endeavour.

Finally, we should mention work of Kuperwasser and Samotij~\cite{ks} on a list version of the Kohayakawa--Kreuter conjecture and of Aigner-Horev, Danon, Hefetz and Letzter~\cite{aigner2022large} on an anti-Ramsey problem in randomly perturbed graphs, which provide potentially interesting directions to extend the work in this paper. In addition, one could consider hypergraph versions of our problem; however, we currently lack an adequate extremal theory of hypergraphs to extend our methods, and this presents us with an obstacle which it would likely be necessary to overcome first before one can study vertex-Ramsey problems in randomly perturbed hypergraphs.


\bibliographystyle{plain}
\bibliography{vertexramseybiblio}
\end{document}